\documentclass[a4paper,11pt,reqno, english]{amsart}  
\usepackage[utf8]{inputenc}
\usepackage[T1]{fontenc}
\usepackage{amsmath,amsthm}
\usepackage{amsfonts,amssymb,enumerate}
\usepackage{url,paralist}
\usepackage{mathtools}  
\usepackage[colorlinks=true,urlcolor=blue,linkcolor=red,citecolor=magenta]{hyperref}
\usepackage{enumerate}
\usepackage{tikz}
\usepackage{dsfont}
\usepackage[left=1in,right=1in,top=1in,bottom=1in]{geometry}

\theoremstyle{plain}
\newtheorem{theorem}{Theorem}[section]
\newtheorem{lemma}[theorem]{Lemma}

\newtheorem{proposition}[theorem]{Proposition}

\theoremstyle{definition}

\newtheorem{remark}[theorem]{Remark}

\newcommand{\Z}{\mathbb{Z}}
\newcommand{\R}{\mathbb{R}}
\newcommand{\C}{\mathbb{C}}

\DeclareMathOperator{\conv}{\mathrm{conv}}

\begin{document}

\title[Roots of real-valued zero mean maps]{Roots of real-valued zero mean maps:\\ Compositions of linear functionals and equivariant maps}


\author{Francesca Cantor \and Julia D'Amico \and Florian Frick \and Eric Myzelev}

\address[FC]{Dept. Math. Stat. Comput. Science, University of Illinois at Chicago, Chicago, IL 60607, USA}
\email{fcant@uic.edu}

\address[JD]{Dept. Math., University of Pennsylvania, Philadelphia, PA 19104, USA}
\email{damico7@sas.upenn.edu}

\address[FF]{Dept. Math. Sciences, Carnegie Mellon University, Pittsburgh, PA 15213, USA}
\email{frick@cmu.edu}

\address[EM]{Dept. Math. Sciences, Carnegie Mellon University, Pittsburgh, PA 15213, USA}
\email{etmyzele@andrew.cmu.edu}

\date{\today}
\maketitle


\begin{abstract}
\small 
We develop a novel topological framework that yields results constraining the distribution of zeros of certain zero mean real-valued maps, namely those obtained from composing a fixed equivariant map with linear functionals. We use this framework to establish upper bounds for the topology of set systems in the domain where (multivariate) trigonometric polynomials do not change their sign, generalizing and, in certain regimes, strengthening results in the literature. Our results more generally contain restrictions on the distribution of zeros of Chebyshev spaces as special cases. Lastly, we apply this framework to derive existence results for efficient cubature rules for compositions of affine functionals and equivariant maps.
\end{abstract}

\section{Introduction}

For a finite set~${S\subset \Z_{> 0}}$, a \emph{trigonometric polynomial with spectrum $S$} is a function of the form 
\[
    f(t)=\sum_{k\in S}(a_k\sin(2\pi kt)+b_k\cos(2\pi kt)),
\]
where for each $k\in S$, $a_k$ and $b_k$ are real numbers that are not both zero. The degree of $f$ is defined as the maximum of~$S$.
Since every such function is periodic, we may consider them as functions from the circle $S^1 = [0,1] / (0 \sim 1)$ of circumference one to~$\R$. Moreover, any such function has \emph{zero mean}, that is, the integral over its domain vanishes. 
Babenko~\cite{babenko1984} showed that any trigonometric polynomial of degree at most~$d$ has a zero on any closed interval of length~$\frac{d}{d+1}$. Our goal here is to develop a topological framework for such ``distribution of zeros'' problems and extend Babenko's result in various ways to a more general class of zero mean maps.

Crucial to Babenko's approach is that the vector space~$T_d$ spanned by the functions 
\[
    \sin(2\pi\cdot kt), 
    \cos(2\pi\cdot kt), \ \text{for} \ k \in \{1,2,\dots, d\}, \ \text{and constant functions}
\]
    forms a \emph{Chebyshev space}, that is, a $(2d+1)$-dimensional vector space of continuous functions $S^1 \to \R$ containing all constant functions such that for all non-constant $f\in V$ and $b \in \R$ the preimage $f^{-1}(b)$ has size at most~$d$. In fact, Babenko proves more general results for Chebyshev spaces, which are important objects of study in approximation theory; see for instance~\cite{schumaker2007}. A curve $\gamma$ in $\R^d$ is called \emph{convex} if every affine hyperplane intersects its image in at most $d$ points. It is easily seen that $T_d$ is a Chebyshev space is equivalent to the convexity of the \emph{trigonometric moment curve} 
\[
    \gamma_d(t) = (\sin 2\pi t, \cos 2\pi t, \sin 2\pi\cdot 2t, \cos 2\pi\cdot 2t, \dots, \sin 2\pi\cdot dt, \cos 2\pi\cdot dt).
\]

The trigonometric moment curve $\gamma_d$ is also a prototypical example of a symmetric curve in addition to being convex. Indeed, it is an orbit of a circle action on Euclidean space that is free away from the origin. Such symmetry is strictly more general than the convexity of the curve (see Lemma~\ref{lem:Zp}), and here we show that Babenko's result extends to the symmetric setting; see Theorem~\ref{thm:equivariant}. In Theorem~\ref{thm:multivariate} we record results for multivariate trigonometric polynomials, which once again extends to the symmetric setting; see Theorem~\ref{thm:Zp-torus}. Similarly, we extend a result of Adams, Bush, and the third author on sign changes of raked trigonometric polynomials on small diameter sets to this symmetric setting; see Theorem~\ref{thm:preimage-vr}. Sign changes of systems of functions are closely related to efficient numerical integration; we state a result about the existence of cubature rules for multivariate trigonometric polynomials and more general functions derived from the symmetric case; see Theorem~\ref{thm:int-eq}. We now carefully state these results.

Trigonometric polynomials of degree~$d$ are compositions of linear functionals with~$\gamma_d$. We generalize results to the setting of compositions of linear functionals with equivariant (i.e., symmetry-preserving) maps. For a group $G$ acting on metric spaces $X$ and~$Y$, a map $f\colon X \to Y$ is \emph{$G$-equivariant} if it commutes with the $G$-actions, that is, if $f(g\cdot x) = g\cdot f(x)$ for all $x\in X$ and for all $g \in G$. The $p$-element cyclic group~$\Z/p$ acts on the circle~$S^1$ by rotations. A point $x\in X$ is a \emph{fixed point} of a $G$-action on~$X$ if $g\cdot x = x$ for all $g\in G$. For $G$-actions on~$\R^d$, we do not necessarily need that $G$ acts by linear maps, only that the action preserves rays through the origin.

\begin{theorem}
\label{thm:equivariant}
    Let $\gamma\colon S^1 \to \R^{2d}$ be a $\Z/p$-equivariant map for some prime $p \ge 2$ and some $\Z/p$-action on~$\R^{2d}$ whose only fixed point is~$0$. Then there is a closed interval $I \subset S^1$ of length~$\frac{d}{d+1}$ such that all compositions of $\gamma$ with linear functionals, that is, maps $f_a \colon S^1 \to \R$ given by $f_a(x) = \langle a, \gamma(x) \rangle$ for $a \in \R^{2d}$, have a zero on~$I$. In fact, the origin is a convex combination of at most $d+1$ points in the image of~$\gamma$.
\end{theorem}

The bound of $\frac{d}{d+1}$ for the length of intervals in Theorem~\ref{thm:equivariant} is optimal since Babenko's original result is. Theorem~\ref{thm:equivariant} guarantees the existence of one interval~$I$ of length~$\frac{d}{d+1}$, where every~$f_a$ has a zero, while Babenko's theorem for trigonometric polynomials asserts this for every such interval; the statement for all intervals follows from Theorem~\ref{thm:equivariant} since by the circle symmetry of~$\gamma_d$ all closed intervals of a given length are the same up to symmetry; see Remark~\ref{rem:homogeneous}.

We derive structural results for the zeros of compositions of linear functionals with $(\Z/p)^n$-equivariant maps from the $n$-torus; see Theorem~\ref{thm:Zp-torus}. Here we first state these results in the special case of
\emph{multivariate trigonometric polynomials}, that is, linear combinations of the functions $(t_1, \dots, t_n) \mapsto \cos(2\pi\sum \alpha_it_i)$ and $(t_1, \dots, t_n) \mapsto \sin(2\pi\sum \alpha_it_i)$ for non-negative integers~$\alpha_i$ that are not all zero. The set $S_f$ of all $n$-tuples of coefficients $(\alpha_1, \dots, \alpha_n)$ is the \emph{spectrum} of the multivariate trigonometric polynomial~$f$. Multivariate trigonometric polynomials are zero mean maps $(S^1)^n \to \R$. Earlier results of this form are due to Kozma and Oravecz~\cite{kozma2002} and Steinerberger~\cite{steinerberger2024}; see Subsection~\ref{subsec:trig}. In the following the torus~$(S^1)^n$ carries the $\ell_2$-metric inherited from the usual Euclidean metric on~$[0,1]^n$. We prove (see Section~\ref{sec:multivariate}):

\begin{theorem}
\label{thm:multivariate}
    Let $n \ge 1$ be an integer, and let $p \ge 3$ be a prime. Let $S \subset (\Z_{\ge 0})^n \setminus \{0\}$ be a set of cardinality at most $\frac12 p^n-1$, and let $r = \sqrt{\frac14(n-1)+(\frac{p-2}{2p})^2}$. Further assume that for every $\alpha\in S$ some coordinate $\alpha_i$ is not divisible by~$p$. Then every multivariate trigonometric polynomial with spectrum $S$ has a zero on any closed geodesic ball of radius~$r$ in~$(S^1)^n$.
\end{theorem}

Such results do not trivially follow from the $1$-dimensional case by restricting the inputs to $(t_1, 0, \dots, 0)$: The restriction of a multivariate trigonometric polynomial to a single non-zero variable is in general not a trigonometric polynomial, but differs from a trigonometric polynomial by an additive constant; in particular, this restriction generally does not have zero mean.

We say that a function $f\colon X\to \R$ \emph{changes sign} on~$A\subset X$ if there are $x,y \in A$ with $f(x) \ge0$ and $f(y)\le0$. We prove the following variant of Theorem~\ref{thm:equivariant} (see Section~\ref{sec:cara}):

\begin{theorem}
\label{thm:preimage-vr}
   Let $\gamma\colon S^1 \to \R^{2d}$ be a $\Z/p$-equivariant map for some prime $p \ge 2$ and some $\Z/p$-action on~$\R^{2d}$ whose only fixed point is~$0$. Then there is a set $A \subset S^1$ of diameter at most~$\frac{d}{2d+1}$ such that all compositions of $\gamma$ with linear functionals, that is, maps $f_a \colon S^1 \to \R$ given by $f_a(x) = \langle a, \gamma(x) \rangle$ for $a \in \R^{2d}$, change sign on~$A$. 
\end{theorem}

In particular, the same holds for compositions of linear functionals with convex curves, that is, for any $(2d+1)$-dimensional Chebyshev space $V \subset C(S^1, \R)$ there is a set $A \subset S^1$ of diameter at most~$\frac{d}{2d+1}$ such that every $f\in V$ with zero mean changes sign on~$A$. For $\gamma(t) = (\cos 2\pi kt, \sin 2\pi kt)_{k \in S}$, $S \subset \Z_{>0}$ of size~$d$, Theorem~\ref{thm:preimage-vr} implies that trigonometric polynomials with frequencies in~$S$ all change sign on some set $A\subset S^1$ of diameter at most~$\frac{d}{2d+1}$.

Not every set $A \subset S^1$ of diameter at most~$\frac{d}{2d+1}$ can be extended to a connected set of diameter at most~$\frac{d}{2d+1}$; the vertices of a regular $(2d+1)$-gon are one example. Thus a change of sign does not imply that there is a set $A$ of diameter at most~$\frac{d}{2d+1}$ where all $f$ have a zero. Theorem~\ref{thm:preimage-vr} was previously established for \emph{raked trigonometric polynomials} of degree at most~$2d-1$, that is, for linear combinations of the functions $t\mapsto \sin(2\pi kt)$ and $t\mapsto \cos(2\pi kt)$ for $k \in \{1,3, \dots, 2d-1\}$; see~\cite{adams2020}. Theorem~\ref{thm:preimage-vr} generalizes this to trigonometric polynomials of other frequencies and--more generally--to compositions of linear functionals with an arbitrary $\Z/p$-equivariant map. Raked trigonometric polynomials also show that this result is optimal~{\cite[Thm.~4]{adams2020}}.

Understanding constraints on sets $A$ where all $f$ change sign is natural, since those are precisely the sets such that evaluation in the points of~$A$ yields formulas for efficient numerical integration; see Proposition~\ref{prop:sign-change}.
We may thus use our topological framework to establish a result concerning efficient numerical integration. A notable recent result employing topological methods for this problem area was the breakthrough of Bondarenko, Radchenko, and Viazovska~\cite{bondarenko2013} giving asymptotically optimal bounds for spherical designs. Recall that a path-connected metric space is \emph{$c$-connected} if the first $c$ homotopy groups vanish.

\begin{theorem}
\label{thm:int-eq}
  Let $X$ be a $c$-connected compact metric space with probability measure~$\mu$ and an action by the nontrivial group~$G$. Let $\gamma \colon X \to \R^n$ be a $G$-equivariant map, where the $G$-action on~$\R^n$ is free on~$\R^n \setminus \{0\}$. Let $V$ be the $(n+1)$-dimensional vector space of compositions of affine functionals with~$\gamma$, that is, of maps $g(x) = \langle a, \gamma(x) \rangle +b$ for $a\in \R^n$ and $b\in \R$. Let $k$ be an integer with $k(c+2) > n$. Then there is a set $A \subset X$ of cardinality at most $k$ and a function $w\colon A\to [0,1]$ such that for every $g \in V$ the integral $\int_{X} g(x) \ d\mu$ is equal to $\sum_{a\in A} g(a)\cdot w(a)$.
\end{theorem}

This improves bounds in the Richter--Tchakaloff cubature theorem on the size of the set~$A$ from $n+1$ to $\lceil\frac{n+1}{c+2}\rceil$, provided that~$V$ consists of compositions of affine functionals with an equivariant map; see Subsection~\ref{subsec:int} for additional context.

The proofs of all of our results use the same basic equivariant-topological approach that works in some generality. In Section~\ref{sec:approach} we provide general results for compositions of linear functionals with a $G$-equivariant map~$\gamma$ from a compact metric space~$X$ to~$\R^n$. We consider a collection~$\mathcal F$ of closed subsets of~$X$. We show that if the space of probability measures with support entirely within a set $A \in \mathcal F$, equipped with a Wasserstein metric of optimal transport, is $(n-1)$-connected, then there is an $A\in \mathcal F$ such that any composition of a linear functional with $\gamma$ changes sign on~$A$. All results advertised above follow from this setup.

\subsection*{Structure of the manuscript}
Our results summarize and generalize Carath\'eodory-type theorems for convex curves, results about the distribution of zeros of (multivariate) trigonometric polynomials, and cubature rules for efficient integration. For this we use standard results from equivariant topology. 
\begin{compactitem}
    \item We give the necessary background and context in Section~\ref{sec:context}. 
    \item We develop our general topological approach in Section~\ref{sec:approach}. There we state and prove Lemma~\ref{lem:cstm-conn} and its variant Lemma~\ref{lem:cstm-eq} that are at the core of our setup and imply all of our main results. The section concludes with a proof of Theorem~\ref{thm:int-eq}. 
    \item The image of a convex closed curve in~$\R^{2d}$ is the orbit of a circle action that is free away from the origin. This is a standard result of algebraic geometry, as we explain in Section~\ref{sec:cara}. There we give an elementary proof of the approximate version for $\Z/p$-actions and $p$ an arbitrarily large prime, which is all that we need to provide a new topological proof of the known result that any point in the convex hull of the image of a closed convex curve is a convex combination of at most $d+1$ points along the curve; see Section~\ref{sec:cara}. Since our proof only uses $\Z/p$-equivariance and not convexity of the curve, our method immediately extends to yield proofs of Theorems~\ref{thm:equivariant} and~\ref{thm:preimage-vr}. 
    \item In Section~\ref{sec:multivariate} we use our framework to prove results for compositions of linear functionals and $(\Z/p)^n$-equivariant maps on the torus. In particular, we prove Theorem~\ref{thm:multivariate} on the distribution of zeros of multivariate trigonometric polynomials.
\end{compactitem}

\section{Context and methods}
\label{sec:context}

Our results constrain the zeros of trigonometric polynomials, generalize Carath\'eodory-type results for convex curves to images of equivariant maps, and consequently yield efficient formulas for numerical integration. Further, our results may be seen as generalizations of Borsuk--Ulam-type theorems to the ``overconstrained'' setting, where generically equivariant maps will not have zeros, yet compositions with linear functionals will exhibit zeros on some set of quantifiable size. Our methods combine Borsuk--Ulam-type results and a construction in metric geometry originally motivated by applications in computational topology and topological data analysis: Metric thickenings. Here we collect the most important context, definitions, notation, and results.

\subsection{Roots of trigonometric polynomials}
\label{subsec:trig}
 A trigonometric polynomial of degree~$d$ is a linear combination of the real and imaginary parts of a complex polynomial of degree~$d$ restricted to the unit circle $S^1 \subset \C$, that is, up to identifying the circle $S^1$ with $[0,1] / (0 \sim 1)$, a function of the form $f(t) = \sum_{k=1}^d a_k\cos(2\pi kt) + b_k\sin(2\pi kt)$, where at least one of $a_d$ or $b_d$ is nonzero. We emphasize that we only consider homogeneous trigonometric polynomials, where we do not allow an additive constant. Thus $\int_0^1 f(t) \ dt = 0$; we say that $f$ is a \emph{zero mean} function. In the following we will always identify $S^1$ with the unit interval $[0,1]$ with its endpoints glued together. In particular, $S^1$ has circumference $1$ and diameter~$\frac12$. Thus the $n$-dimensional torus $(S^1)^n$ is obtained from the cube $[0,1]^n$ by gluing opposite facets together without a twist. The torus thus inherits a metric from the usual Euclidean metric on~$[0,1]^n$, and in particular the torus has diameter~$\frac12\sqrt{n}$.

 \begin{theorem}[Babenko~\cite{babenko1984}]
 \label{thm:babenko}
    Let $I \subset S^1$ be a closed interval of length~$\frac{d}{d+1}$. Then any trigonometric polynomial of degree~$d$ has a zero on~$I$.
 \end{theorem}

 Also see Gilbert and Smyth~\cite{gilbert2000} for further results. Tabachnikov~\cite{tabachnikov1997} exploits the cyclic symmetry of trigonometric polynomials to give upper bounds for the measure of the set, where a trigonometric polynomial may be non-positive. Ismailov~\cite{ismailov2023} proved generalizations to arbitrary sets of frequencies and generalizes further beyond trigonometric polynomials. Variants of these results for multivariate trigonometric polynomials have been studied:

  \begin{theorem}[Kozma and Oravecz~\cite{kozma2002}]
  \label{thm:koz-ora}
    Let $S \subset \Z^n \setminus \{0\}$ be finite with $-S = S$. Let $f\colon [0,1]^n \to \R$ be given by $f(x) = \sum_{\lambda \in S} c(\lambda)\exp(2\pi i\langle x, \lambda\rangle)$. The induced map $\widehat f\colon (S^1)^n \to \R$ has a zero on any closed geodesic ball of radius~$\sum_{\lambda \in S} \frac{1}{4\sqrt{\langle \lambda, \lambda\rangle}}$. 
 \end{theorem}

  \begin{theorem}[Steinerberger~\cite{steinerberger2024}]
  \label{thm:steinerberger}
    Let $S \subset \Z^n \setminus \{0\}$ be finite with $-S = S$. Let $f\colon [0,1]^n \to \R$ be given by $f(x) = \sum_{\lambda \in S} c(\lambda)\exp(2\pi i\langle x, \lambda\rangle)$. The induced map $\widehat f\colon (S^1)^n \to \R$ has a zero on any closed geodesic ball of radius~$n^{3/2}\sum_{\lambda \in \Lambda} \frac{1}{\lambda}$, where $\Lambda = \{\sqrt{\langle \lambda, \lambda\rangle} \mid \lambda\in S\}$.
 \end{theorem}

 Recall that the diameter of the $n$-torus $(S^1)^n$ is~$\frac12\sqrt{n}$. While these results give strong bounds for large frequencies, the bounds are trivial in other cases. Our results provide nontrivial bounds independent of the spectrum, and in particular improve on Theorem~\ref{thm:steinerberger} for any spectrum~$S$, where $\sum_{\lambda \in \Lambda} \frac{1}{\lambda} > \frac{1}{2n}$ with~$\Lambda$ as above.

\subsection{Borsuk--Ulam results} 
Let $S^d$ denote the $d$-dimensional sphere. The classical Borsuk--Ulam theorem~\cite{borsuk1933} states that any continuous map $\gamma\colon S^d \to \R^d$ that is \emph{odd}, i.e.~$\gamma(-x) = -\gamma(x)$ for all $x \in S^d$, must have a zero.

The following is a simple rephrasing of the Borsuk--Ulam theorem: Let $V$ be a $d$-dimensional vector space of continuous odd maps $f\colon S^d \to \R$. Then there is an $x_0 \in S^d$ with $f(x_0) = 0$ for all $f\in V$. In other words, the functions in~$V$ have a common zero. This rephrasing follows from the classical statement by choosing a basis $f_1, \dots, f_d$ of~$V$ and defining $\gamma(x) = (f_1(x), \dots, f_d(x))$. Now if $V$ is an $n$-dimensional vector space of continuous odd maps $f\colon S^d \to \R$ for $n > d$ then the maps in $V$ will generically not have a common zero. In this case, there is a set $A \subset S^d$ of diameter at most~$\frac{n}{2n+2}$ such that all $f \in V$ change sign on~$A$; see~\cite{adams2020, crabb2023}.

Numerous generalizations of the Borsuk--Ulam theorem for maps respecting larger groups of symmetries have been proven. Let $G$ be a group and let $X$ and $Y$ be topological spaces with $G$-actions by homeomorphisms. Recall that a map $f\colon X \to Y$ is \emph{$G$-equivariant} if $f(g\cdot x) = g\cdot f(x)$ for all $x\in X$ and for all $g\in G$, and that a $G$-action on $X$ is \emph{free} if $g\cdot x \ne x$ for all $x\in X$ and all non-identity $g \in G$. A $G$-action on a CW complex~$X$ is \emph{cellular} if for every $g \in G$ and every cell $\sigma$ of~$X$, $g\cdot \sigma$ is again a cell of~$X$. A topological space $X$ is \emph{$n$-connected} for $n \ge 0$ if the first $n$ homotopy groups of $X$ vanish. We will need the following two results:

\begin{theorem}[Dold~\cite{dold1983}]
\label{thm:dold}
    Let $G$ be a nontrivial finite group that acts cellularly on the CW complex~$X$ and acts freely on the sphere~$S^n$. If there is a $G$-equivariant map $f\colon X \to S^n$ then $X$ is not $n$-connected. 
\end{theorem}

Dold states this result as if there is a $G$-equivariant map $S^m \to S^n$ for free $G$-actions on domain and codomain, then $m \ge n$. This is equivalent to the version above: If $X$ is homotopically $n$-connected, then by elementary obstruction theory there is a $G$-equivariant map $S^{n+1} \to X$, which composed with~$f$ yields a $G$-equivariant map $S^{n+1} \to S^n$. Here we may assume that the action on $X$ is free, since $f$ maps to a free space.
Recall that a point $x\in X$ is a \emph{fixed point} of a $G$-action on~$X$ if $g\cdot x = x$ for all $g\in G$. Not having fixed points is generally a strictly weaker requirement than the freeness of a $G$-action. 

\begin{theorem}[Volovikov~\cite{volovikov1996}]
\label{thm:volov}
    Let $p\ge2$ be a prime, and let $G = (\Z/p)^k$ be an elementary abelian group that acts on the CW complex $X$ and without fixed points on~$S^n$. If there is a $G$-equivariant map $f\colon X \to S^n$ then $X$ is not $n$-connected.
\end{theorem}

\subsection{Carath\'eodory-type results}
\label{subsec:cara}

We refer to Barvinok~\cite{barvinok2002} and Matou\v sek~\cite{matousek2013} for the basics of convexity. We recall the classical theorems of Carath\'eodory and Steinitz:

\begin{theorem}[Carath\'eodory~\cite{caratheodory1907, caratheodory1911}]
\label{thm:carath}
    Let $X\subset \R^d$ be a subset, and let $x \in \conv X$. Then there are $x_1, \dots, x_{d+1} \in X$ with $x \in \conv \{x_1, \dots, x_{d+1}\}$.
\end{theorem}

For $X \subset \R^d$ the \emph{Carath\'eodory number} is the smallest integer~$k$ such that any $x\in \conv X$ is a convex combination of at most $k$ points in~$X$.

\begin{theorem}[Steinitz~\cite{steinitz1913}]
\label{thm:steinitz}
    Let $X\subset \R^d$ be a subset, and let $x$ be in the interior of~$\conv X$. Then there are $x_1, \dots, x_{2d} \in X$ such that $x$ is in the interior of~$\conv \{x_1, \dots, x_{2d}\}$.
\end{theorem}

We point out that the papers of Carath\'eodory and Steinitz are foundational papers in convexity, and properties of convex sets were established for the purpose of studying trigonometric polynomials and Fourier series. To this end the \emph{trigonometric moment curve} $\gamma(t) = (\sin t, \cos t, \dots, \sin dt, \cos dt)$ appears in Carath\'eodory's work; the convex hull of the image of~$\gamma$ is now referred to as \emph{Carath\'eodory orbitope}~\cite{sanyal2011}. More generally in~\cite{sanyal2011} Sanyal, Sottile, and Sturmfels study \emph{orbitopes}, that is, convex hulls of orbits of actions of compact Lie groups on Euclidean space. Longinetti, Sgheri, and Sottile~\cite{longinetti2010} determine bounds for the Carath\'eodory number of the convex hull of orbits of a certain $SO(3)$-action on~$\R^{10}$. Instead of studying the convex hull of the orbit of a single representation, that is, of a $G$-equivariant map $G \to \R^d$, our work more generally bounds Carath\'eodory numbers for convex hulls of the images of equivariant maps; see~\cite{adams2020}, where the connection between Borsuk--Ulam-type theorems and Carath\'eodory numbers of orbitopes exploited in the present paper is made.

A curve $\gamma \colon [0,1] \to \R^d$ is called \emph{convex} if every affine hyperplane in~$\R^d$ intersects $\gamma$ in at most $d$ points, that is, if the equation $\langle a, \gamma(x) \rangle = b$ has at most $d$ solutions for fixed $a \in \R^d \setminus \{0\}$ and~$b\in \R$. It is easy to see that closed convex curves, that is, those with $\gamma(0) = \gamma(1)$, only exist for even~$d$. For the following result see for example~{\cite[Prop.~II.9.4]{barvinok2002}}; we will give a topological proof, see Section~\ref{sec:cara}.

\begin{theorem}[Carath\'eodory~\cite{caratheodory1911}]
\label{thm:conv-curve}
    Let $\gamma \colon S^1 \to \R^{2d}$ be a closed convex curve, and let $y \in \conv \{\gamma(x)  \mid  x \in S^1\}$. Then there are $x_1, \dots, x_{d+1} \in S^1$ with $y \in \conv \{\gamma(x_1), \dots, \gamma(x_{d+1})\}$.
\end{theorem}

Carath\'eodory proves this result for $\gamma = \gamma_d$, the trigonometric moment curve, and mentions without proof that any closed convex curve may be transformed into~$\gamma_d$ by a similarity. More is true: A convex curve in~$\R^d$ is a normal curve, that is, a $1$-dimensional variety of degree~$d$; extending~$\R^d$ to $d$-dimensional projective space, normal curves are projectively unique~{\cite[Prop.~18.9]{harris2013algebraic}}. Let $\gamma\colon I \to \R^d$ be a convex curve. Then the coordinate functions of~$\gamma$ together with a constant non-zero function span a Chebyshev space (and every Chebyshev space is of this form). A central fact about Chebyshev spaces is that the coordinate functions of $\gamma$ together with constant functions span a Chebyshev space if and only if for all $x_1 < \dots < x_d \in I$ the determinants $\det(\gamma(x_1) \ \gamma(x_2) \ \cdots \ \gamma(x_d))$ are all positive or all negative; see~{\cite[Ch.~1, \S~4]{karlin}}. In convex-geometric terms, this can be phrased as the points $\gamma(x_1), \dots, \gamma(x_d)$ are in cyclic position~\cite{barany2014, sturmfels1987cyclic}.

The hyperplane separation theorem (or finite-dimensional Hahn--Banach theorem) states that for any two disjoint convex closed sets $C, D\subset \R^d$, where at least one of them is compact, there is an affine hyperplane that strictly separates them, that is, there is an $a\in \R^d$ and $b \in \R$ such that $\langle a,x\rangle < b < \langle a,y \rangle$ for all $x \in C$ and all $y\in D$; see~\cite[Thm.~1.2.4]{matousek2013}.

\begin{lemma}
\label{lem:hahn-banach}
    Let $V$ be a vector space of continuous maps $f \colon X \to \R$, where $X$ is compact, and let $f_1, \dots, f_d$ be a basis of~$V$. Let $\gamma\colon X \to \R^d$ be given by $\gamma(x) = (f_1(x), \dots, f_d(x))$, and let $A\subset X$ be closed. Then $0 \in \conv(\gamma(A))$ if and only if every $f \in V$ changes sign on~$A$.
\end{lemma}

\begin{proof}
    We have that $0 \notin \conv(\gamma(A))$ if and only if there is an $a \in \R^{2d}$ such that $\langle a, \gamma(x) \rangle > 0$ for all $x \in A$. Since $f_a(x) = \langle a, \gamma(x) \rangle = \sum_i a_if_i(x)$ is a linear combination of the $f_i$ this is equivalent to $f_a$ not changing its sign on~$A$.
\end{proof}

\subsection{Integration by finite evaluation}
\label{subsec:int}

Carath\'eodory-type results are closely related to numerical integration by averaging over a finite sample. A standard example for this is the Richter--Tchakaloff theorem below~\cite{tchakaloff1957formules, schmudgen2020ten}, which follows from Carath\'eodory's theorem.

\begin{theorem}[Richter--Tchakaloff]
\label{thm:tchakaloff}
    Let $(X, \mu)$ be a probability space, and let $V \subset L^1(X, \mu)$ be a $k$-dimensional subspace. Then there is a set $A\subset X$ of cardinality~$k$ and a function $w\colon A \to [0,1]$ such that
    \[
        \int_X f \ d\mu = \sum_{a \in A} w(a)f(a) \ \text{for all} \ f \in V.
    \]
\end{theorem}

A strengthening of this for univariate polynomials reducing the size of the set~$A$ is Gau{\ss}--Legendre quadrature, that is, the following double-precision integration formula; see~{\cite[Ch.~II, 10]{barvinok2002}}.

\begin{theorem}[Gau{\ss}--Legendre]
\label{thm:double-precision}
    Let $\rho\colon [0,1] \to [0, \infty)$ be a continuous density function. For any positive integer~$m$ there are points $\tau_1^*, \dots, \tau_m^* \in [0,1]$ and $\lambda_1, \dots, \lambda_m > 0$ such that 
    \[
        \int_0^1 f(t)\rho(t) \ dt = \sum_{i=1}^m \lambda_if(\tau_i^*)
    \]
    for any polynomial of degree at most $2m-1$.
\end{theorem}

Here the name ``double-precision'' refers to the phenomenon that a sample of $m$ points is sufficient to integrate polynomials of degree up to~$2m-1$, a vector space of dimension~$2m$. This follows directly from the fact that every point in the convex hull of the moment curve $(t,t^2, \dots, t^{2m-1})$ is in the convex hull of $m$ points along this curve.
Similarly, Theorem~\ref{thm:int-eq} gives criteria under which the size of the set~$A$ in the Richter--Tchakaloff theorem may be reduced. It asserts the existence of an integration formula with $A$ of size roughly~$\frac{k}{c+2}$, where $c$ is the connectivity of the domain, provided that the zero mean part of~$V$ has a basis $(f_1, \dots, f_{k-1})$ that together form a $G$-equivariant map to~$\R^{k-1}$, where the $G$-action on~$\R^{k-1}$ is free away from the origin. 

For a measure space $X$ with measure~$\mu$ of finite total measure and $V$ a finite-dimensional vector space of integrable functions $X\to \R$, an \emph{averaging set} is a finite set $A \subset X$ such that $\frac{1}{\mu(X)}\int_X f \ d\mu = \frac{1}{|A|}\sum_{a\in A} f(a)$ for all $f\in V$. Seymour and Zaslavsky show the existence of averaging sets in broad generality by relating it to Carath\'eodory-type results.
By Lemma~\ref{lem:hahn-banach}, we can moreover relate integration by finite evaluation results to results about functions changing their sign on the given point sample:

\begin{proposition}
\label{prop:sign-change}
    Let $X$ be a metric measure space of total measure~$1$, and let $V$ be a finite-dimensional vector space of continuous functions $X \to \R$. Suppose that all constant functions are in~$V$. Let $A \subset X$ be a closed set. Then every $f \in V$ with zero mean changes sign on~$A$ if and only if there is a function $w\colon A \to [0,1]$ with finite support such that for every $g \in V$ we have that
    \[
        \int_X g(x) \ dx = \sum_{a \in A} w(a)g(a).
    \]
\end{proposition}

\begin{proof}
    First suppose that every $f\in V$ with $\int_X f(x) \ dx = 0$ changes sign on~$A$. Let $f_1, \dots, f_n$ be a basis of the codimension one subspace of $V$ that contains all zero mean functions. Let 
    \[
        \gamma\colon X \to \R^n, \ \gamma(x) = (f_1(x), \dots, f_n(x)). 
    \]
    By Lemma~\ref{lem:hahn-banach} we have that $0 \in \conv(\gamma(A))$. By Theorem~\ref{thm:carath} there are $x_1, \dots, x_{n+1} \in A$ and $\lambda_1, \dots, \lambda_{n+1} \ge 0$ with $\sum_i \lambda_i=1$ such that $0 = \sum_i \lambda_i\gamma(x_i)$. Let $g\in V$ be arbitrary. Then there are coefficients $a_1, \dots, a_n \in \R$ and a constant $c\in \R$ such that $g = c+\sum_i a_if_i$. Then
    \[\sum_i \lambda_ig(x_i) = \sum_i \lambda_i\left(c+\sum_j a_j f_j(x_i)\right) = c \sum_i \lambda_i = c = c+ \sum_i a_i \int_X f_i(x) \ dx = \int_X g(x) \ dx.\]
    Setting $w(x_i) = \lambda_i$ and $w(a) = 0$ for $a \in A \setminus \{x_1, \dots, x_{n+1}\}$ proves one direction of the proposition. (Here we assume for simplicity of notation that the $x_i$ are pairwise distinct.)

    Conversely, let $w\colon A \to [0,1]$ be a function with $\int_X g(x) \ dx = \sum_{a\in A} w(a)g(a)$ for all $g\in V$. Since $V$ contains non-zero constant functions, $w$ cannot be identically zero. Given any $f\in V$ with zero mean, we have that $0 = \int_X f(x) \ dx = \sum_{a\in A} w(a)f(a)$. Since $w(a) \ge 0$ for all $a \in A$, the function $f$ must change sign on~$A$.
\end{proof}

For $X = S^d \subset \R^{d+1}$ and $V$ the vector space of real polynomials in $d+1$ variables of total degree at most~$t$, an averaging set $A\subset S^d$ is called \emph{spherical $t$-design}. Asymptotically optimal bounds for the size of $t$-designs were proven by Bondarenko, Radchenko, and Viazovska~\cite{bondarenko2013} using topological techniques.

\subsection{Metric thickenings}

Let $(X,d)$ be a compact metric space. Let $\mathcal F$ be a set of subsets of~$X$. Denote by $\mathcal P(X, \mathcal F)$ the space of Borel probability measures that are supported in some $A \in \mathcal F$ equipped with the $1$-Wasserstein metric of optimal transport. The \emph{$1$-Wasserstein distance} between Borel probability measures $\mu$ and $\nu$ on~$X$ is the infimum of $\int_{X\times X} d(x,y) \ d\pi$ over all Borel probability measures~$\pi$ on $X\times X$ with $\mu(B) = \pi(B \times X)$ and $\nu(B) = \pi(X \times B)$ for all Borel sets $B \subset X$. We will not make reference to metric properties of~$\mathcal P(X, \mathcal F)$; the choice of $1$-Wasserstein metric is thus of no consequence, since any $q$-Wasserstein metric induces the same topology. We call $\mathcal P(X, \mathcal F)$ the \emph{metric thickening} by $\mathcal F$ of~$X$. We refer to~\cite{adamaszek2018, adams2024} for the basics. There the definition is restricted to sets $\mathcal F$ that only consist of finite sets, and thus $\mathcal P(X, \mathcal F)$ only contains convex combinations of \emph{Dirac measures}~$\delta_x$ that assign measure one exactly to those measurable sets that contain~$x$. The more general definition here will simplify the presentation somewhat.

A \emph{simplicial complex} $\Sigma$ is a set of finite sets closed under taking subsets, that is, for $\sigma \in \Sigma$ and $\tau \subset \sigma$ we have that $\tau \in \Sigma$. Singletons $\{v\} \in \Sigma$ are \emph{vertices}, whereas general $\sigma \in \Sigma$ are \emph{faces}. Every simplicial complex $\Sigma$ is naturally a topological space: Let $V$ be the set of vertices of~$\Sigma$ and identify elements of $V$ with their indicator functions $\mathds 1_v$ in~$\R^V$. Endow~$\R^V$ with the supremum-norm and consider 
\[
    \bigcup_{\sigma \in \Sigma} \conv\{\mathds 1_v \mid v \in \sigma\} \subset \R^V,
\]
the \emph{geometric realization} of~$\Sigma$. Whenever we refer to topological properties of~$\Sigma$, they are properties of this geometric realization. We refer to Kozlov~\cite{kozlov2007} and Matou\v sek~\cite{matousek2003} for the basics.

\begin{theorem}[Adams, Frick, and Virk~\cite{adams2023}]
\label{thm:gillespie}
    Let $X$ be a compact metric space. Let $\mathcal U$ be an open cover of~$X$, and let $\mathcal F$ be a simplicial complex of finite subsets of sets in~$\mathcal U$. Then $\mathcal F$ and $\mathcal P(X, \mathcal F)$ have isomorphic homotopy groups.
\end{theorem}

Gillespie~\cite{gillespie2024} showed the stronger result that in this case $\mathcal F$ and $\mathcal P(X, \mathcal F)$ are weakly homotopy equivalent. Recall that a topological space $X$ is \emph{$n$-connected} for $n \ge 0$ if the first $n$ homotopy groups of $X$ vanish. The maximal~$n$ such that $X$ is $n$-connected will be denoted by~$\mathrm{conn}(X)$.

\begin{theorem}[Adamaszek and Adams~\cite{adamaszek2017}]
\label{thm:vr}
    Let $\frac{k}{2k+1} < r < \frac{k+1}{2k+3}$ for some integer~$k\ge1$. Let $\mathcal F$ be the simplicial complex of all finite subsets $X \subset S^1$ of diameter less than~$r$. Then $\mathcal F$ is homotopy equivalent to~$S^{2k+1}$, and in particular $\mathcal F$ is $2k$-connected.
\end{theorem} 

For $X$ and $Y$ topological spaces, the \emph{join} $X * Y$ is the quotient of $X \times Y \times [0,1]$ by the equivalence relation generated by $(x,y,0) \sim (x',y,0)$ and $(x,y,1)\sim(x,y',1)$ for $x,x'\in X$ and $y,y' \in Y$. This operation can be iterated and is associative, so that for example the threefold join $X * Y * Z$ is well-defined. Denote the $k$-fold join of~$X$ by~$X^{*k}$. Elements of $X^{*k}$ can uniquely be referred to as abstract convex combinations $\lambda_1x_1 \oplus \dots \oplus \lambda_kx_k$, where $x_i \in X$, $\lambda_i \ge 0$, and $\sum_i \lambda_i = 1$. In this notation, if some coefficient $\lambda_i = 0$, then the corresponding $x_i$ has no influence on the point in~$X^{*k}$.

\begin{lemma}[{\cite[Prop.~4.4.3]{matousek2003}}]
\label{lem:join-conn}
     Let $X$ and $Y$ be simplicial complexes. Then $\mathrm{conn}(X*Y) \ge \mathrm{conn}(X)+\mathrm{conn}(Y)+2$.
\end{lemma}

Let $X$ be a compact metric space with a $G$-action. Then the $k$-fold join~$X^{*k}$ carries the diagonal $G$-action given by $g\cdot (\lambda_1x_1 \oplus \dots \oplus \lambda_kx_k) = \lambda_1g\cdot x_1 \oplus \dots \oplus \lambda_kg\cdot x_k$. Let $\mathcal F$ be a set of subsets of~$X$ such that for all $A\in \mathcal F$ and for all $g\in G$ we have that $g\cdot A \in \mathcal F$. In this case we say that $\mathcal F$ is \emph{invariant under the $G$-action}. The $G$-action on $X$ canonically lifts to $\mathcal P(X, \mathcal F)$ via $g\cdot \mu(A) = \mu(g^{-1}\cdot A)$.

\section{The topological approach}
\label{sec:approach}

A configuration space / test map scheme is a popular approach to reduce not necessarily topological problems to a topological question. The configuration space parametrizes the space of all potential solutions for the problem, while the test map measures to which extent a potential solution deviates from an actual one. The goal then is to use properties of the test map (such as the map respecting certain inherent symmetries) to show that the test map must have a zero on the configuration space. This implies the existence of a solution.

In this section we develop a configuration space / test map scheme that can be used to establish that for a given finite-dimensional vector space~$V$ of zero mean continuous functions $X \to \R$ and for a given set~$\mathcal F$ of subsets of~$X$ there is an $A \in \mathcal F$ such that every $f\in V$ changes sign on~$A$. In particular, if $A$ is connected, then every $f \in V$ has a zero in~$A$. The strength of these results, that is, the largest dimension of a vector space $V$ exhibiting this phenomenon, is governed by the topology of the simplicial complex obtained as the downward-closure of the finite subsets of the set system~$\mathcal F$. For compact metric space~$X$ and an open cover~$\mathcal F$, this is up to weak homotopy equivalence the space~$\mathcal P(X, \mathcal F)$ of probability measures with support fully contained in some $A\in \mathcal F$ and equipped with an optimal transport metric; see~\cite{gillespie2024}.

Let $V$ be a finite-dimensional vector space of zero mean continuous functions $X \to \R$, and let $f_1, \dots, f_n \colon X \to \R$ be a basis of~$V$.
Let $\gamma \colon X \to \R^n$ be the map $\gamma(x) = (f_1(x), \dots, f_n(x))$. By integration we get a map
\[
    \Gamma\colon \mathcal P(X, \mathcal F) \to \R^n, \ \mu \mapsto \int_X \gamma(x) \ d\mu.
\]

\begin{lemma}
\label{lem:cstm}
    Let $X$ be a compact metric space, and let $\mathcal F$ be a set of closed subsets of~$X$. With notation as above, $\Gamma$ has a zero if and only if there is an $A \in \mathcal F$ such that every $f\in V$ changes sign on~$A$. In particular, if every $A \in\mathcal F$ is connected, then $\Gamma$ has a zero if and only if there is an $A \in \mathcal F$ such that every $f\in V$ has a zero in~$A$.
\end{lemma}

\begin{proof}
    By Proposition~\ref{prop:sign-change} every $f\in V$ changes sign on~$A$ if and only if there is a finitely-supported (and non-zero) $w\colon A\to [0,1]$ with $0 = \int_X f(x) \ dx = \sum_{a\in A} w(a)f(a)$ for all~$f \in V$. We may assume that $\sum_{a\in A} w(a) = 1$, and in fact the $w$ constructed in the proof of Proposition~\ref{prop:sign-change} has this property without any need for normalization. Let $\mu = \sum_{a\in A} w(a)\delta_a$. Then
    \[
        \int_X f_j(x) \ d\mu = \sum_{a \in A} w(a)f_j(a) = 0.
    \]
    Thus, $\Gamma(\mu) = 0$.

    Conversely, let $\mu \in \mathcal P(X, \mathcal F)$ with $\Gamma(\mu) = 0$. By definition of Riemann integral $0$ is in the closure of the convex hull of~$\gamma(S)$, where $S$ denotes the support of~$\mu$. There is a closed set $A \supset S$ with $A \in \mathcal F$. Then since $\gamma(A)$ is compact, so is its convex hull, and thus $0 \in \conv(\gamma(A))$. By Theorem~\ref{thm:carath} there are $x_1, \dots, x_{n+1} \in A$ and $\lambda_1, \dots, \lambda_{n+1} \ge 0$ with $\sum_i \lambda_i = 1$ such that $0 = \sum_i \lambda_i \gamma(x_i)$. Define $w\colon A \to [0,1]$ by $w(x_i) = \lambda_i$ and $w(a) = 0$ for $a \in A \setminus \{x_1, \dots, x_{n+1}\}$. Then $0 = \sum_{a \in A} w(a)f_j(a)$ and thus $0 = \sum_{a \in A} w(a)f(a)$ for all $f\in V$. As remarked above, this is sufficient. 

    The second part of the lemma follows immediately from the first part and the intermediate value theorem.
\end{proof}

\begin{remark}
\label{rem:homogeneous}
    Let $X$ be a compact metric space that is homogeneous, that is, there is a group~$G$ acting by isometries on~$X$ such that for any $x,y\in X$ there is a $g\in G$ with $g\cdot x = y$. Let $\mathcal F$ be the set of closed geodesic balls in~$X$ of some fixed radius~$r$. Let $\gamma\colon X\to \R^n$ be $G$-equivariant, where $0$ is the only fixed point of the $G$-action on~$\R^n$. If for $\Gamma$ as above, $\Gamma(\mu) = 0$ for $\mu \in \mathcal P(X, \mathcal F)$ supported in some geodesic ball~$B$ of radius~$r$, then by homogeneity for any geodesic ball~$B'$ there is $\mu' \in \mathcal P(X, \mathcal F)$ supported in~$B'$ such that $\Gamma(\mu') = 0$. In particular, for a homogeneous metric space~$X$ and $G$-equivariant $\gamma\colon X\to \R^n$, if Lemma~\ref{lem:cstm} guarantees that all $f \in V$ have a zero on some geodesic ball of radius~$r$, then every $f \in V$ has a zero on every geodesic ball of radius~$r$.
\end{remark}

For $G$-equivariant $\gamma\colon X\to \R^n$, Lemma~\ref{lem:cstm} combined with Borsuk--Ulam theorems for larger symmetry groups yields the following configuration space / test map scheme for detecting common zeros of compositions of linear functionals with~$\gamma$ on some set $A\in \mathcal F$ in terms of the topology of~$\mathcal P(X, \mathcal F)$.

\begin{lemma} 
\label{lem:cstm-conn}
    Let $G$ be a nontrivial finite group. Assume that $G$ acts on the compact metric space~$X$ and acts on~$\R^n$ such that the action on $\R^n\setminus \{0\}$ is free (or fixed-point free if $G = (\Z/p)^k$ for some prime~$p$). Let $\gamma \colon X \to \R^n$ be $G$-equivariant. Let $\mathcal F$ be a set of connected closed subsets of~$X$ that is invariant under the $G$-action. If $\mathcal P(X, \mathcal F)$ is $(n-1)$-connected, then there is an $A\in\mathcal F$ such that all compositions of $\gamma$ with a linear functional, that is, maps $f_a \colon X \to \R$ given by $f_a(x) = \langle a, \gamma(x) \rangle$ for $a \in \R^n$, have a zero on~$A$.
\end{lemma}

\begin{proof}
    Let
\[
    \Gamma\colon \mathcal P(X, \mathcal F) \to \R^n, \ \mu \mapsto \int_X \gamma(x) \ d\mu.
\]
    For $\mu \in \mathcal P(X, \mathcal F)$ with finite support, say $\mu = \sum_{i=1}^{n}\lambda_i \delta_{x_i}$, the evaluation $\Gamma(\mu) = \sum_i \lambda_i\gamma(x_i)$ is a convex combination of~$\gamma(x_i)$. For $g \in G$ we have that 
    \[
    \Gamma(g\cdot \mu) = \sum_i \lambda_i\gamma(g\cdot x_i)
    = g\cdot \sum_i \lambda_i\gamma(x_i) = g\cdot \Gamma(\mu).
    \]
    Thus by approximation, since $\gamma$ is $G$-equivariant so is~$\Gamma$. The map~$\Gamma$ has a zero by Theorem~\ref{thm:dold}, or by Theorem~\ref{thm:volov} if the $G$-action is only fixed-point free and $G=(\Z/p)^k$.

    By Lemma~\ref{lem:cstm} there is an $A \in \mathcal F$ such that every linear combination of the coordinates of~$\gamma$, that is, every map $f_a \colon X \to \R$ given by $f_a(x) = \langle a, \gamma(x) \rangle$ for $a \in \R^n$, changes sign on~$A$. Then every $f_a$ has a zero on~$A$, since $A$ is connected.
\end{proof}

The proof of Lemma~\ref{lem:cstm-conn} also shows the following version that we will sometimes find more convenient:

\begin{lemma}
\label{lem:cstm-eq}
    Let $G$ be a nontrivial finite group. Assume that $G$ acts on the compact metric space~$X$ and acts on~$\R^n$ such that the action on $\R^n\setminus \{0\}$ is free (or fixed-point free if $G = (\Z/p)^k$ for some prime~$p$). Let $\gamma \colon X \to \R^n$ be $G$-equivariant. Let $\mathcal F$ be a set of connected closed subsets of~$X$ that is invariant under the $G$-action. If there is a $G$-equivariant map $Y \to \mathcal P(X, \mathcal F)$, but no $G$-equivariant $Y \to S^{n-1}$, then there is an $A\in\mathcal F$ such that all compositions of $\gamma$ with a linear functional, that is, maps $f_a \colon X \to \R$ given by $f_a(x) = \langle a, \gamma(x) \rangle$ for $a \in \R^n$, have a zero on~$A$.
\end{lemma}

As a consequence of this setup, we can prove the $(c+2)$-fold precision integration formula of Theorem~\ref{thm:int-eq}. Here $c$ refers to the homotopical connectivity of the domain. 

\begin{proof}[Proof of Thm.~\ref{thm:int-eq}]
    By Proposition~\ref{prop:sign-change} we have to show that there is a set $A \subset X$ of cardinality~$k$ such that every $g \in V$ changes sign on~$A$. Consider as before 
    \[
        \Gamma\colon \mathcal P(X, \mathcal F) \to \R^n, \ \mu \mapsto \int_X \gamma(x) \ d\mu,
    \]
    where $\mathcal F$ contains all sets of cardinality at most~$k$. By Lemma~\ref{lem:cstm} we have to show that $\Gamma$ has a zero. If not, $\Gamma$ induces a $G$-equivariant map $\mathcal P(X, \mathcal F) \to S^{n-1}$. The $k$-fold join~$X^{*k}$ is $(kc+2k-2)$-connected (by Lemma~\ref{lem:join-conn}) and the map 
    \[
        \Phi\colon X^{*k} \to \mathcal P(X, \mathcal F), \ \Phi(\lambda_1x_1 \oplus \dots \oplus \lambda_kx_k) = \sum_i \lambda_i\delta_{x_i}
    \]
    is $G$-equivariant. By composition with $\Gamma$ we get a $G$-equivariant map $X^{*k} \to S^{n-1}$, in contradiction to Theorem~\ref{thm:dold}.
\end{proof}

Theorem~\ref{thm:int-eq} guarantees that functions in an $(n+1)$-dimensional vector space of maps $X \to \R$ can be exactly integrated by evaluation in about $\frac{n}{c+2}$ many points of~$X$. Even for $c=0$, that is, for path-connected~$X$, this yields double-precision integration formulas as in Theorem~\ref{thm:double-precision}. For highly connected~$X$, by which we mean a $(c+1)$-dimensional (e.g., as a simplicial complex) and $c$-connected~$X$, the space~$\mathcal P(X, \mathcal F)$ of probability measures supported in at most~$k$ points of~$X$, has dimension $k(c+2)-1$. If $k(c+2)-1 < n$ or equivalently $k(c+2) \le n$, then generically the map~$\Gamma$ will not have zeros, and so for highly connected~$X$ Theorem~\ref{thm:int-eq} is generically optimal.

\section{The topology of Carath\'eodory-type results for convex curves}
\label{sec:cara}

Recall that an embedding $\gamma\colon S^1 \to \R^d$ (also called a \emph{closed curve}) is \emph{convex} if every affine hyperplane in~$\R^d$ intersects it in at most~$d$ points. As explained in Subsection~\ref{subsec:cara} a closed convex curve is up to similarity the trigonometric moment curve, and in particular the orbit of a circle action on Euclidean space. This is already claimed in Carath\'eodory's paper~\cite{caratheodory1911}, and is a standard result in algebraic geometry~{\cite[Prop.~18.9]{harris2013algebraic}} as well as in an equivalent formulation in the language of Chebyshev spaces~{\cite[Ch.~1, \S~4]{karlin}}. Here we first provide an elementary proof of the weaker result that any closed convex curve can be arbitrarily well approximated by a $\Z/p$-equivariant map from the circle. 

\begin{lemma}
\label{lem:Zp}
    Let $\gamma \colon S^1 \to \R^{2d}$ be a closed convex curve such that the origin is in the interior of~$\conv(\gamma(S^1))$, and let $\varepsilon > 0$. Then there is a sufficiently large prime~$p$ and $p$ points $x_1, \dots, x_p$ on~$\gamma$ such that there is a $\Z/p$-action on~$\R^{2d}$ that respects rays, which is free on~$\R^d \setminus \{0\}$ and such that $x_1, \dots, x_p$ is an orbit of this action. Further, the $x_i$ can be chosen such that every point in the image of~$\gamma$ is at distance at most~$\varepsilon$ from some~$x_i$.
\end{lemma}

\begin{proof}
    Let $p \ge 4d$ be a prime. Since the origin is in the convex hull of~$\gamma(S^1)$, by Steinitz' theorem, Theorem~\ref{thm:steinitz}, there are $p$ points $x_1, \dots, x_p$ such that the origin is in the interior of~$\conv \{x_1, \dots, x_p\}$. By potentially increasing~$p$ we can enlarge the set of points $x_i$ such that every $y \in \gamma(S^1)$ is at distance at most~$\varepsilon$ from some~$x_i$. These points are in convex position: Otherwise we could find an $x_j$ and and a set $I \subset [p] \setminus \{i\}$ of size $2d+1$ such that $x_j \in \conv \{x_i \ \mid \ i \in I\}$. The point $x_j$ is necessarily in the interior of the simplex $\conv \{x_i \ \mid \ i \in I\}$, since otherwise $2d+1$ points would lie on a common hyperplane. Let $k,\ell \in I$ such that no other point $x_i$ for $i \in I$ lies between $x_k$ and~$x_\ell$ along~$\gamma$. Continuously move $x_k$ towards $x_\ell$ along~$\gamma$, that is, $y(t)$ is a point on $\gamma$ that continuously depends on~$t$ with $y(0) = x_k$, $y(1) = x_\ell$ and $y(t) \ne x_i$ for all $t \in (0,1)$ and all $i \in I$. By the above $x_j$ will remain in the interior of $\conv(\{x_i \ \mid \ i \in I \setminus \{k\}\} \cup \{y(t)\})$ throughout. For $t=1$ this implies that the $2d+1$ points $x_j$ and $\{x_i \ \mid \ i \in I \setminus \{k\}\}$ are contained in a common hyperplane, a contradiction.

    Thus $P = \conv\{x_1, \dots, x_p\}$ is a simplicial polytope with the origin in its interior. A $\Z/p$-action on the vertices of this polytope is given by $x_i \mapsto x_{i+1}$ with indices modulo~$p$. This induces a unique action on the boundary of~$P$ by affinely extending to faces. Lastly, by declaring that $0$ is a fixed point of the action, the $\Z/p$-action extends to the cones over faces of the boundary of~$P$ and thus to all of~$\R^{2d}$. Since no proper face of $P$ is setwise fixed, this action is free away from~$0$.
\end{proof}

We will first provide an equivariant-topological proof of Theorem~\ref{thm:conv-curve} and as a consequence of Theorem~\ref{thm:babenko}. Recall that Theorem~\ref{thm:conv-curve} states that for a closed convex curve $\gamma \colon S^1 \to \R^{2d}$, any point in~$\conv(\gamma(S^1))$ is a convex combination of at most $d+1$ points along~$\gamma$.

\begin{proof}[{Proof of Thm.~\ref{thm:conv-curve}}]
    Let $x_0 \in \conv(\gamma(S^1))$. We may assume that $x_0$ is in the interior of this convex hull; the general case follows by approximation and compactness of~$\gamma(S^1)$. By translation we assume that $x_0 = 0$. 
    
    Let $\varepsilon > 0$. Let $p$ be a sufficiently large prime, and $x_1, \dots, x_p$ points along~$\gamma$ chosen according to Lemma~\ref{lem:Zp} such that the map $f\colon S^1 \to \R^{2d}$ that agrees with $\gamma$ on~$\gamma^{-1}(J)$, where $J$ is the interval along $\gamma$ from~$x_1$ to~$x_2$, and otherwise is $\Z/p$-equivariant deviates from $\gamma$ pointwise by at most~$\varepsilon$. Define the $\Z/p$-equivariant map
    \[
        F \colon (S^1)^{*(d+1)} \to \R^{2d}, \ \lambda_1x_1 \oplus \dots \oplus \lambda_px_p \mapsto \lambda_1f(x_1) + \dots + \lambda_pf(x_p).
    \]
    Since $(S^1)^{*(d+1)}$ is $2d$-connected by Lemma~\ref{lem:join-conn} (in fact it is homeomorphic to~$S^{2d+1}$), the map $F$ has a zero: If $F$ did not have a zero, it would induce a map $S^{2d+1} \to S^{2d-1}$ by normalization, which will again be equivariant since the $\Z/p$-action on~$\R^{2d}$ respects rays. This would contradict Theorem~\ref{thm:dold}. This implies that there are $y'_1, \dots, y'_{d+1} \in f(S^1)$ with $0 \in \conv\{y'_1, \dots, y'_{d+1}\}$. Let $y_i$ be a point on $\gamma(S^1)$ at distance less than $\varepsilon$ from~$y'_i$. Then $\conv\{y_1, \dots, y_{d+1}\}$ is at distance less than $\varepsilon$ from~$0$. Letting $\varepsilon$ go to $0$ finishes the proof by compactness of~$S^1$.
\end{proof}

We point out that Theorem~\ref{thm:babenko} follows easily from Theorem~\ref{thm:conv-curve}; in fact, we will prove the following generalization, also due to Babenko~\cite{babenko1984}:

\begin{theorem}[Babenko]
\label{thm:babenko-chebyshev}
    Let $V \subset C(S^1,\R)$ be a $(2d+1)$-dimensional Chebyshev space. Then there is a closed interval~$I \subset S^1$ of length~$\frac{d}{d+1}$ such that every $f \in V$ with zero mean has a zero on~$I$.
\end{theorem}

\begin{proof}
    Let $f_1, \dots, f_{2d}$ be a basis for the codimension one subspace of zero mean maps in~$V$, and let $\gamma\colon S^1 \to \R^{2d}$ be given by $\gamma(x) = (f_1(x), \dots, f_{2d}(x))$. Let $a \in \R^{2d} \setminus \{0\}$ and let $b \in \R$. The set of solutions of the equation $\langle a, \gamma(x)\rangle = b$ is the preimage~$f_a^{-1}(b)$, where $f_a = \sum_i a_if_i$, and thus has size at most~$2d$. Thus $\gamma$ is a convex curve. By Theorem~\ref{thm:conv-curve}, there are $d+1$ points $t_1, \dots, t_{d+1} \in S^1$ such that $0 \in \conv \{\gamma(t_1), \dots, \gamma(t_{d+1})\}$. These $d+1$ points are necessarily contained in a closed interval~$I$ of length~$\frac{d}{d+1}$. By Lemma~\ref{lem:hahn-banach} all zero mean $f \in V$ change sign on~$I$. Thus by continuity all zero mean $f\in V$ have a zero on~$I$.
\end{proof}

Theorem~\ref{thm:babenko} that trigonometric polynomials of degree at most~$d$ have a zero on \emph{any} closed interval $I \subset S^1$ of length~$\frac{d}{d+1}$ follows from Theorem~\ref{thm:babenko-chebyshev} by Remark~\ref{rem:homogeneous}, that is, because $\gamma$ as in the proof above is actually an orbit of an isometric circle action on~$\R^{2d}$.

In the proof of Theorem~\ref{thm:conv-curve} convexity was only used to approximate $\gamma$ by $\Z/p$-equivariant maps. In particular, the result remains true for $\Z/p$-equivariant maps from the circle, a much larger class of maps than closed convex curves. Since the Carath\'eodory-type result Theorem~\ref{thm:conv-curve} was the key ingredient in the proof of Theorem~\ref{thm:babenko-chebyshev} this implies that results about the distribution of zeros in Chebyshev spaces may be extended to spaces that consists of compositions of affine functionals and $\Z/p$-equivariant maps, whereas Chebyshev spaces consist of compositions of affine functionals with convex curves. Theorems~\ref{thm:equivariant} and~\ref{thm:preimage-vr} are those generalizations. We prove them now.

\begin{proof}[Proof of Theorem~\ref{thm:equivariant}]
    Let $\mathcal F$ be the set of all closed intervals of length~$\frac{d}{d+1}$ in~$S^1$. Let 
    \[
        \Phi\colon (S^1)^{*(d+1)} \to \mathcal P(S^1, \mathcal F), \ 
        \Phi(\lambda_1x_2 \oplus \dots \oplus \lambda_{d+1}x_{d+1}) = \sum_i \lambda_i\delta_{x_i}.
    \]
    Any probability measure in the image of~$\Phi$ is supported in at most $d+1$ points, and in particular, this support is a subset of some interval $A \in \mathcal F$. The join $(S^1)^{*(d+1)}$ is $2d$-connected by Lemma~\ref{lem:join-conn}, in fact, it is homeomorphic to~$S^{2d+1}$, and thus does not admit a $\Z/p$-equivariant map to~$\R^{2d}\setminus\{0\} \simeq S^{2d-1}$ by Theorem~\ref{thm:dold}. Now use Lemma~\ref{lem:cstm-eq}.
\end{proof}

\begin{remark}
    \label{rem:trigpolys}
    Fix a prime $p$ and let $f\colon S^1\to \R$ be a function whose Fourier expansion converges uniformly to~$f$. Theorem~\ref{thm:equivariant} yields that if $\widehat f(pk)=0$ for all $k\in \Z$, then $f$ has a zero on every interval of length~$(p-1)/p$. In fact, if $S\subset \Z_+$ is the spectrum of~$f$, then $f$ has a zero in every closed interval of length $\frac{r}{r+1}\le \frac{p-1}{p}$, where $r=|\{k \textrm{ mod } p\mid k\in S\}|$. To see this, first note that it suffices by uniform converge to consider the case where the Fourier expansion is a trigonometric polynomial, that is, $f$ is a trigonometric polynomial with spectrum~$S$ such that every $k \in S$ is nonzero modulo~$p$. Define $S_j=\{k\in S\mid k\textrm{ mod } p \equiv j\}$. Note that we can identify $\R^{2r}$ with $\C^{r}$ and then define $\gamma\colon S^1\to \C^r$ by setting component $i$ to be $\gamma(t)_i=\sum_{k\in S_j}\widehat f(k)e^{2\pi i kt }$. Then $\gamma$ is $\Z/p$-equivariant, and the result follows from Theorem~\ref{thm:equivariant} by setting $a$ to be a vector with $1$ in every component. See Tabachnikov~\cite{tabachnikov1997} for related results.
\end{remark}

\begin{proof}[Proof of Theorem~\ref{thm:preimage-vr}]
    Let $\mathcal F$ be the simplicial complex of all finite subsets of~$S^1$ of diameter less than~$r$ for some $r > \frac{d}{2d+1}$.   By Lemma~\ref{lem:cstm} we have to show that $\Gamma\colon \mathcal P(S^1, \mathcal F) \to \R^{2d}$ with $\Gamma(\mu) = \int_{S^1} \gamma(x) \ d\mu$ has a zero. This implies that for every $r > \frac{d}{2d+1}$ there is an $A \subset S^1$ of diameter at most~$r$, where every $f\in V$ changes sign. The statement of the theorem is derived by taking the limit $r \to \frac{d}{2d+1}$ and using that $S^1$ is compact. By Theorem~\ref{thm:gillespie} $\mathcal P(S^1,\mathcal F)$ is $2d$-connected since $\mathcal F \simeq S^{2d+1}$ (by Theorem~\ref{thm:vr}) is. Since $\Gamma$ is $\Z/p$-equivariant (as in the proof of Lemma~\ref{lem:cstm}), it has a zero by Theorem~\ref{thm:dold}, which completes the proof.
\end{proof}

\section{Multivariate trigonometric polynomials \\ and equivariant maps from the torus}\label{sec:multivariate}

To establish bounds for the size of geodesic balls, where multivariate trigonometric polynomials must have zeros, we need lower bounds for the equivariant topology of the metric thickening of $(S^1)^n$ by~$\mathcal F$, where $\mathcal F$ consists of all geodesic balls of a fixed radius~$r$. For this it suffices to give a $(\Z/p)^n$-invariant subset~$L$ of~$(S^1)^n$ such that for any proper subset of~$L$ the geodesic balls of radius~$r$ around these points have a common intersection.

\begin{lemma}
\label{lem:torus-grid}
    Let $n\ge1$ be an integer, and let $p\ge2$ be a prime. Let $r=\sqrt{\frac14(n-1)+(\frac{p-2}{2p})^2}$. Then there is a $(\Z/p)^n$-invariant subset $L\subset (S^1)^n$ such that for every $x \in L$ the set $L\setminus \{x\}$ lies in a common geodesic ball of radius~$r$.
\end{lemma}

\begin{proof}
    Consider the $n$-torus $(S^1)^n$ as the cube $[0,1]^n$ with opposite faces identified without a twist. Let 
    \[
        \widehat L = \left\{\left(\frac{k_1}{p}, \dots, \frac{k_n}{p}\right) \in [0,1]^n \mid k_1, \dots, k_n \in \{0,1,\dots,p-1\}\right\}.
    \]
    The set $\widehat L$ is a grid of $p^n$ points in~$[0,1]^n$, and after taking the quotient yields a set $L$ of size $p^n$ in~$(S^1)^n$. We compute the distance $r$ from the barycenter~$(\frac12, \dots, \frac12)$ to the point~$(0,\dots,0,\frac{1}{p})$. Here it is irrelevant whether we compute this distance in $[0,1]^n$ or in~$(S^1)^n$. This distance is $r = \sqrt{\frac14(n-1)+(\frac{p-2}{2p})^2}$. The geodesic ball of radius $r$ around~$(\frac12, \dots, \frac12)$ contains all points of~$L$ with the exception of~$(0, \dots, 0)$. By symmetry for every $x \in L$, the set $L\setminus \{x\}$ is contained in some closed geodesic ball of radius~$r$.
\end{proof}

We can now show that compositions of linear functionals with a $(\Z/p)^n$-equivariant map $(S^1)^n \to \R^d$ have a common zero on a geodesic ball whose radius may be quantified in terms of $d$ and~$n$.

\begin{theorem}
\label{thm:Zp-torus}
    Let $n \ge 1$ be an integer, let $p\ge 3$ be a prime, let $r = \sqrt{\frac14(n-1)+(\frac{p-2}{2p})^2}$, and let $d = p^n-3$. Let $\gamma \colon (S^1)^n \to \R^d$ be $(\Z/p)^n$-equivariant, where the $(\Z/p)^n$-action on~$\R^d$ has no fixed points other than~$0$. Then there is a closed geodesic ball $B \subset (S^1)^n$ of radius~$r$ such that all compositions of $\gamma$ with a linear functional, that is, maps $f_a \colon (S^1)^n \to \R$ given by $f_a(x) = \langle a, \gamma(x) \rangle$ for $a \in \R^d$, have a zero on~$B$.
\end{theorem}

\begin{proof}
    Let $L \subset (S^1)^n$ be the set constructed in Lemma~\ref{lem:torus-grid}. Let $\Sigma$ be the simplicial complex of all proper subsets of~$L$. Thus $\Sigma$ is the boundary of a simplex on $p^n$ vertices and thus a sphere of dimension~$p^n-2$. As $L$ is $(\Z/p)^n$-invariant, $\Sigma$ carries a $(\Z/p)^n$-action inherited from the natural action on its vertices. Let $\mathcal F$ be the set of all closed geodesic balls of radius~$r$ in~$(S^1)^n$. There is a $(\Z/p)^n$-equivariant map $\Phi\colon \Sigma \to \mathcal P((S^1)^n, \mathcal F)$: Every point~$x$ in~$\Sigma$ is a convex combination of all but one of the points in~$L$, say $x = \sum_{y \in L\setminus \{y_0\}} \lambda_yy$. Define $\Phi(x) = \sum_{y \in L\setminus \{y_0\}} \lambda_y\delta_y$. The set $L \setminus \{y_0\}$ is contained in a common closed geodesic ball of radius~$r$, and thus $\Phi$ is well-defined. There is no $(\Z/p)^n$-equivariant map $\Sigma \to S^{d-1}$ by Theorem~\ref{thm:volov}. Now use Lemma~\ref{lem:cstm-eq}.
\end{proof}

\begin{proof}[Proof of Theorem~\ref{thm:multivariate}]
    Let $S \subset (\Z_{\ge 0})^n \setminus \{0\}$ be a set of cardinality at most $\frac12 p^n-1$. Since $p$ is odd, $2|S| \le p^n-3$. For $\alpha = (\alpha_1, \dots, \alpha_n) \in S$ define $g_\alpha\colon (S^1)^n \to \R^2$ by 
    \[
        g_\alpha(t_1, \dots, t_n) = \left(\cos\left(2\pi\sum_i \alpha_it_i\right), \sin\left(2\pi\sum_i \alpha_it_i\right)\right).
    \]
    That is, $g_\alpha$ parametrizes the orbit of the $(S^1)^n$-action on~$\C$ given by $(z_1, \dots, z_n) \cdot z = z_1^{\alpha_1}\dots z_n^{\alpha_n}z$. 
    Since $z_1^{\alpha_1}\dots z_n^{\alpha_n} = 1$ for every choice of $p$th roots of unity $z_1, \dots, z_n$ only if all $\alpha_i$ are divisible by~$p$, the $(\Z/p)^n$-action on $\R^2$ given by $(z_1, \dots, z_n) \cdot z = z_1^{\alpha_1}\dots z_n^{\alpha_n}z$ is fixed-point free on~$\R^2\setminus \{0\}$.

    Let $\gamma\colon (S^1)^n \to \R^{2|S|}$ be given by $\gamma(z) = (g_\alpha(z))_{\alpha \in S}$. The map $\gamma$ is equivariant with respect to the $(\Z/p)^n$-action on $\R^{2|S|}$ given (for each~$\R^2$ separately) above, and the standard action on~$(S^1)^n$ as a subgroup. Theorem~\ref{thm:Zp-torus} guarantees the existence of a ball of radius $r = \sqrt{\frac14(n-1)+(\frac{p-2}{2p})^2}$, where all maps $f_a(x) = \langle a, \gamma(x) \rangle$ for $a \in \R^{2|S|}$, that is, all multivariate trigonometric polynomials with spectrum~$S$, vanish. This is sufficient by Remark~\ref{rem:homogeneous} since $(S^1)^n$ is homogeneous. 
\end{proof}

\begin{remark}
    For $r > 0$ let $\mathcal F_r$ be the set of all closed geodesic balls of radius~$r$ in~$(S^1)^n$. Let $\Sigma$ be the simplicial complex of all finite subsets of~$\mathcal F_r$. This complex is the nerve complex (or \v Cech complex) of closed geodesic balls of radius~$r$ in the $n$-torus. Our proof of Theorem~\ref{thm:multivariate} shows that if $\Sigma$ is $c$-connected then any multivariate trigonometric polynomial with spectrum of size at most $\frac{c+1}{2}$ has a zero on any closed geodesic ball of radius~$r$. Thus topological lower bounds for~$\Sigma$ provide results for zeros of multivariate trigonometric polynomials. In the contrapositive, if some $n$-ary trigonometric polynomial with spectrum~$S$ is positive on some closed geodesic ball of radius~$r$, then the metric thickening $\mathcal P((S^1)^n, \mathcal F_r)$ admits a $(\Z/p)^n$-equivariant map to~$\R^{2|S|} \setminus \{0\}$, and in particular is not $(2|S|-1)$-connected.

    Motivated by the optimality of Theorem~\ref{thm:equivariant}, we conjecture that an $(S^1)^n$-equivariant map $\mathcal P((S^1)^n, \mathcal F_r) \to S^{2d-1}$ for some fixed-point free action on~$S^{2d-1}$ exists if and only if every $n$-ary trigonometric polynomial with spectrum of size~$d$ has a zero on every closed geodesic ball of radius~$r$. This is known for $n=1$.
\end{remark}

\section*{Acknowledgements}

The authors were supported by NSF grant DMS 2042428. We are grateful to Nikola Sadovek and Ran Tao for helpful conversations. We are grateful to an anonymous referee for helpful guidance.


\end{document}